\documentclass[12pt,reqno]{amsart}
\usepackage{amssymb}
\usepackage{eucal}
\newcommand{\R}{\mathbb R}

\newcommand{\Z}{\mathbb Z}

\newcommand{\T}{\mathbb T}

\newcommand{\secao}[1]{\section{#1}\setcounter{equation}{0}}

\newtheorem{theorem}{\sc Theorem}[section]

\newtheorem{remark}[theorem]{\sc Remark}

\title[Ill-posedness for the BBM]{On the ill-posedness result for the BBM equation}

\author[Mahendra Panthee]{Mahendra Panthee}

\address{Centro de Matem\'atica, Universidade do Minho\\4710-057, Braga, Portugal} 

\email{mpanthee@math.uminho.pt}

\thanks{This work was partially supported by the Research Center of 
Mathematics of the University of Minho, Portugal through the FCT Pluriannual Funding Program.}


\begin{document}

\maketitle

\begin{abstract}
We prove that the initial value problem (IVP) for the BBM equation is ill-posed for data in $H^s(\R)$, $s<0$ in the sense that the flow-map $u_0\mapsto u(t)$ that associates to initial data $u_0$ the solution $u$  cannot be continuous at the origin from $H^s(\R)$ to even $\mathcal{D}'(\R)$  at any fixed $t>0$ small enough. This result is sharp.
\end{abstract}
\secao{Introduction}
Let us  consider the initial value problem (IVP)
\begin{equation}\label{ivp1}
\begin{cases}
u_{t}-u_{xxt}+u_x+uu_x =0,\quad x, t \in \R,\\
u(x,0)=u_0(x)
\end{cases}
\end{equation}
where $u=u(x,t)$ is a  real valued function. This model describes the propagation of one-dimensional, unidirectional small amplitude long waves in nonlinear dispersive media \cite{BBM}. This model, widely known as the Benjamin-Bona-Mahony (BBM) equation, is the regularized counterpart of the Korteweg-de Vries (KdV) equation and is extensively studied in the recent literature, see for example \cite{AABCW}, \cite{ABS}, \cite{BPS}, \cite{BT}  and references therein.

For simplicity and later use, let us define  $D_x := \frac1i \partial_x$.  Then the IVP \eqref{ivp1} can be written in the following form
\begin{equation}\label{ivp2}
\begin{cases}
iu_{t}=\varphi(D_x)u+\frac12 \varphi(D_x)u^2\\
u(x,0)=u_0(x),
\end{cases}
\end{equation}
where $\varphi(D_x)$ is the Fourier multiplier operator defined by 
\begin{equation}\label{phi-d}
\widehat{\varphi(D_x)u(\xi)} = \varphi(\xi)\hat{u}(\xi), \quad {\text{with}}\quad \varphi(\xi) = \frac{\xi}{1+\xi^2}.
\end{equation}

 Recently, Bona and Tzvetkov \cite{BT} proved that the IVP \eqref{ivp1} is globally well-posed in the $L^2$-based Sobolev spaces $H^s(\R)$ for $s\geq 0$. More precisely, they proved the following result.
 
\begin{theorem}\cite{BT}\label{well-posed}
Fix $s\geq 0$. For any $u_0\in H^s(\R)$, there exists a $T=T(\|u_0\|_{H^s})>0$ and a unique solution $u\in C([0, T]; H^s(\R))$ of the IVP \eqref{ivp1}.

Moreover, for $R>0$, let $\mathcal{B}_R$ denote the ball of radius $R$ centered at the origin in $H^s(\R)$ and let $T=T(R)>0$ denote a uniform existence time for the IVP \eqref{ivp1} with $u_0\in \mathcal{B}_R$. Then the correspondence $u_0\mapsto u$ that associates to $u_0$ the solution $u$ of the IVP \eqref{ivp1} with initial data $u_0$ is a real analytic mapping of $\mathcal{B}_R$ to $C([-T, T]; H^s(\R))$.
\end{theorem}
 
 This Theorem improves the earlier known results by Benjamin et. al. \cite{BBM} where the IVP \eqref{ivp1} was shown to be globally well-posed  for data in $H^k$, $k\in \Z$ and $k\geq 1$.  Moreover, the authors in \cite{BT} also proved that the IVP \eqref{ivp1} for given data in $H^s(\R)$, $s<0$ is  ill-posed in the sense that, the flow map $u_0 \mapsto u(t)$ is not even $C^2$. The exact ill-posedness result proved in \cite{BT} reads as follows.
 
\begin{theorem}\cite{BT}\label{illposed-1}
For any $s<0$, $T>0$ the flow-map $u_0\mapsto u(t)$ established in Theorem \ref{well-posed} is not of class $C^2$ from $H^s$ to $C([0, T]; H^s(\R))$.
\end{theorem}
 
 This type of ill-posedness notion was first introduced by Bourgain \cite{Bo2} to show that the well-posedness result for the KdV 
equation in $H^s(\R),$ $s>-3/4$, is essentially optimal if one strengthens 
the usual notion of well-posedness by requiring the flow-map
$$
\phi\mapsto u_{\phi}(t), \quad \quad |t| < T,
$$
should act smoothly (for eg. $C^3$) on the space under consideration
(instead of just continuous). This notion of well-posedness seems to
be natural because, if one uses the contraction mapping principle to solve
the integral equation associated with the Cauchy problem, the flow-map
acts smoothly from $H^s$ to itself (see for eg.,  \cite{Bo1}, \cite{KPV3} and \cite{KPV2}).
Motivated by the work of Bourgain  \cite{Bo2}, Takaoka \cite{TAKA}
showed that the nonlinear Schr\"{o}dinger equation with derivative in a
nonlinear term is ill-posed in $H^s(\R),$ $s<1/2$. Utilizing the same
techniques Tzvetkov \cite{Tz} proved that the KdV equation is locally
ill-posed in ${H}^s(\R)$ for $s<-3/4$ if one requires only $C^2$
regularity of the flow-map in the notion of well-posedness.  Several dispersive models are proved to be ill-posed in certain function spaces using this notion, see for example \cite{GPS}, \cite{LP}, \cite{Mo}, \cite{MR}, \cite{MST-1}, \cite{MST-2} and references therein.

 Quite recently, motivated by the idea introduced by Bejenaru and Tao \cite{BejT}, Molinet and Vento \cite{MV} obtained a sharp ill-posedness result for the KdV-Burgers equation in $H^s$, $s<-1$ in the sense that the flow-map $u_0\mapsto u(t)$ cannot be continuous from $H^s(\R)$ to even $\mathcal{D}'(\R)$ at any fixed $t>0$ small enough. 
Following the scheme presented in Molinet and Vento \cite{MV}, we prove that the   IVP \eqref{ivp1} is ill-posed in the sense that the flow-map $u_0\mapsto u(t)$ is discontinuous at the origin from $H^s(\R)$ to even $\mathcal{D}'(\R)$, for any $s<0$. More precisely, we prove the following result.
\begin{theorem}\label{illposed-2}
Let $s<0$, then the IVP \eqref{ivp1} is ill-posed in $H^s(\R)$ in the following sense: there exists 
$T > 0$ such that for any $0 < t < T$, the flow-map $u_0 \mapsto u(t)$ constructed in Theorem \ref{well-posed} is discontinuous at the origin from $L^2(\R)$ endowed with the topology inducted by $H^s(\R)$ into $\mathcal{D}'(\R)$.
\end{theorem}

 This result improves the ill-posedness result proved in Bona and Tzvetkov \cite{BT} and is sharp.

The remainder of this article is organized as follows. In the rest of this section we will introduce the notations that will be used throughout this work. In Section \ref{sec2} we will sketch the proof of the earlier ill-posedness result, viz. Theorem \ref{illposed-1}. Finally, Section \ref{sec3} is devoted to supply the proof of the main result of this work, viz. Theorem \ref{illposed-2}. \\

\noindent {\bf{Notation}}:
We denote by  $\mathcal{F}_x(f)(\xi)$ or $\hat{f}(\xi)$, the Fourier transform of $f$ in $x$ 
variable
\begin{equation*}
\mathcal{F}_x(f)(\xi)\equiv \hat{f}(\xi) := \frac1{\sqrt{2\pi}}\int_{\R}\,e^{-ix\xi}f(x)\,dx.
\end{equation*}
We use $H^s(\R)$ to denote the $L^2$-based Sobolev space of order $s$ with norm

\begin{equation*}
\|f\|_{H^s(\R)} =\Big(\int_{\R}(1+|\xi|^2)^{s}
|\hat{f}(\xi)|^2\,d\xi\Big)^{1/2}.
\end{equation*}
Various constants whose exact values are
immaterial will be denoted by $C$. Finally, we use the notation $A \lesssim B $ if there exists
a constant $C>0$ such that $A < C B$, $A\gtrsim B$ if there exists a
constant $C>0$ such that $A > C B$ and $A\sim B $ if $A\lesssim B $
and $A \gtrsim B$.\\

\secao{Earlier Ill-posedness Result}\label{sec2}

In this section we will describe, in brief, the ideas presented in \cite{BT} to prove the ill-posedness result stated in Theorem \ref{illposed-1}.

 For $u_0\in H^s(\R)$, consider the Cauchy problem
\begin{equation}\label{ivp3}
\begin{cases}
iu_{t}=\varphi(D_x)u+\frac12 \varphi(D_x)(u^2)\\
u(x,0)=\epsilon u_0(x),
\end{cases}
\end{equation}
where $\epsilon >0$ is a parameter. The solution $u^{\epsilon}(x,t)$ of \eqref{ivp3} depends on the parameter $\epsilon$. We can write \eqref{ivp3} in the equivalent integral equation form as
\begin{equation}\label{int-1}
u^{\epsilon}(x,t) = \epsilon S(t)u_0(x) - \frac i2\int_0^t S(t-t')\varphi(D_x))(u^2(x,t'))dt',
\end{equation}
where, $S(t) = e^{-it\varphi(D_x)}$ is the unitary group describing the solution of the linear part of the IVP \eqref{ivp3}.

Differentiating $u^{\epsilon}(x,t)$ in \eqref{int-1} with respect $\epsilon$ and evaluating at $\epsilon =0$ we get
\begin{equation}\label{int-2}
\frac{\partial u^{\epsilon}(x,t)}{\partial \epsilon}\Big|_{\epsilon=0} = S(t)u_0(x) =:u_1(x)
\end{equation}
and
\begin{equation}\label{int-3}
\frac{\partial^2 u^{\epsilon}(x,t)}{\partial \epsilon^2}\Big|_{\epsilon=0} = -i\int_0^t S(t-t')\varphi(D_x)(u_1^2(x,t'))dt' =:u_2(x).
\end{equation}

If the flow-map is $C^2$ at the origin from $H^s(\R)$ to $C([-T, T];H^s(\R))$, we must have
\begin{equation}\label{eq-bilin}
\|u_2\|_{L_T^{\infty}H^s(\R)}\lesssim \|u_0\|_{H^s(\R)}^2.
\end{equation}

The main idea to complete the proof of Theorem \ref{illposed-1} is to find an appropriate initial data $u_0$ for which the estimate \eqref{eq-bilin} fails to hold. For this, Bona and Tzvetkov \cite{BT} considered the following initial data defined via the Fourier transform
\begin{equation}\label{ex-1}
\widehat{u_0}(\xi) = \gamma^{-\frac12}N^{-s}[\chi_{I}(\xi) +\chi_{I}(-\xi)],
\end{equation}
where $I = [N-\gamma, N+\gamma]$ with $N\gg 1$ and $\gamma =N^{-\sigma}$, for $0<\sigma \ll 1$.
For this particular initial data, the estimate \eqref{eq-bilin} fails to hold for any $s<0$ thereby finishing the proof. 

In this work we will renormalize the above example and exploit the analyticity of the flow-map obtained in Theorem \ref{well-posed} to prove the main result, Theorem \ref{illposed-2}.

\section{ Proof of the Sharp Ill-posedness Result}\label{sec3}

In this section we provide the proof of the main result of this work. As discussed earlier, we will  consider the renormalized form of the counter example constructed in Bona and Tzvetkov \cite{BT} and follow the scheme introduced in \cite{BejT} and \cite{MV} to accomplish the proof.

\begin{proof}[Proof of Theorem \ref{illposed-2}]
Let $N\gg 1$ and define $\phi_N$ via the Fourier transform as 
\begin{equation}\label{counter-1}
\widehat{\phi_N}(\xi) = \chi_{I_N}(\xi) +\chi_{I_N}(-\xi),
\end{equation}
where $I_N = [N-1, N+1]$.

Simple calculation shows that $\|\phi_N\|_{L^2(\R)} \sim 1$ and $\|\phi_N\|_{H^s(\R)}\to 0$, for any $s<0$.

As pointed out in the previous section, the second iteration in the Picard scheme is the following,
\begin{equation}\label{eq2.1}
I_2(h,h, t)= \int_0^tS(t-t') \varphi(D_x)[S(t')h]^2dt'.
\end{equation}

Now, using $\phi_N$ in place of $h$ and computing the Fourier transform in $x$, we obtain
\begin{equation}\label{eq2.2}
\begin{split}
\mathcal{F}_x(I_2(\phi_N,\phi_N, t))(\xi)&= \int_0^t\mathcal{F}_x[S(t-t') \varphi(D_x)[S(t')\phi_N]^2](\xi)dt'\\
&=\int_0^t e^{-i(t-t')\varphi(\xi)}\varphi(\xi)\mathcal{F}_x[S(t')\phi_N]^2 (\xi)dt'\\
&=\int_0^t e^{-i(t-t')\varphi(\xi)}\varphi(\xi) \int_{\R}e^{-it'\varphi(\xi_1)}\widehat{\phi_N}(\xi_1)e^{-it'\varphi(\xi-\xi_1)}\widehat{\phi_N}(\xi-\xi_1)d\xi_1dt'\\
&= \int_{\R} e^{-it\varphi(\xi)}\widehat{\phi_N}(\xi_1)\widehat{\phi_N}(\xi-\xi_1)\varphi(\xi) \int_0^te^{-it'\theta(\xi, \xi_1)}dt'd\xi_1,
\end{split}
\end{equation}
where, 
\begin{equation}\label{def-theta}
\theta(\xi, \xi_1):= \varphi(\xi_1)+\varphi(\xi-\xi_1)-\varphi(\xi) = \frac{\xi\xi_1(\xi-\xi_1)(\xi^2-\xi\xi_1+\xi_1^2+3}{(1+\xi_1^2)[1+(\xi-\xi_1)^2](1+\xi^2)}.
\end{equation}

Therefore, in view of \eqref{counter-1}
\begin{equation}\label{eq2.3}
\begin{split}
\mathcal{F}_x(I_2(t, \phi_N,\phi_N))(\xi)&= e^{-it\varphi(\xi)}i\varphi(\xi)\int_{\R} \widehat{\phi_N}(\xi_1)\widehat{\phi_N}(\xi-\xi_1) \frac{e^{-it\theta(\xi, \xi_1)} -1}{\theta(\xi, \xi_1)}d\xi_1\\
&=e^{-it\varphi(\xi)}i\varphi(\xi)\int_{\xi_1\in I_N\cup (-I_N)\atop{\xi-\xi_1\in I_N\cup (-I_N)}}\frac{e^{-it\theta(\xi, \xi_1)} -1}{\theta(\xi, \xi_1)}d\xi_1.
\end{split}
\end{equation}

Now we move to find a lower bound for $\|I_2(\phi_N,\phi_N,t)\|_{H^s}$. With the similar reasoning as in \cite{BT}, \cite{GPS} and \cite{Tz}, the main contribution to this norm comes from the combination of frequencies such that $|\theta(\xi, \xi_1)|$ is small. In fact,
\begin{equation}\label{eq2.4}
\begin{split}
\|I_2(\phi_N,\phi_N,t)\|_{H^s(\R)}^2 &\geq \int_{-\frac14}^{\frac14}(1+|\xi|^2)^s|\mathcal{F}_x(I_2( \phi_N,\phi_N,t))(\xi)|^2d\xi\\
&\geq \int_{-\frac14}^{\frac14}(1+|\xi|^2)^s|\varphi(\xi)|^2\Big|\int_{A_{\xi}} \frac{e^{-it\theta(\xi, \xi_1)} -1}{\theta(\xi, \xi_1)}d\xi_1\Big|^2d\xi,
\end{split}
\end{equation}
where 
$$A_{\xi} := \{\xi_1 : \xi_1\in I_N,\; \xi-\xi_1\in -I_N\quad {\text or},\quad \xi-\xi_1\in I_N,\; \xi_1\in -I_N\}.$$

Note that, for $\xi_1\in A_{\xi}$ we have $|\xi_1|\sim |\xi-\xi_1|\sim N$, $|\xi|\sim {\text{o}}(1)$ and  consequently $|\theta(\xi, \xi_1)|\sim {\text{o}}(1)$ . Therefore, for any $t>0$ fixed, we get
\begin{equation}\label{eq2.5}
\Big|\frac{e^{-it\theta(\xi, \xi_1)} -1}{\theta(\xi, \xi_1)}\Big|\geq C|t|.
\end{equation}

Also note that, for $|\xi|\sim {\text{o}}(1)$, with $\xi_1\in A_{\xi}$, ${\text{measure}}(A_{\xi}) \geq 1$. Hence, for any fixed $t>0$ and for some positive constant $C_0$, using \eqref{eq2.5} in \eqref{eq2.4}, we obtain
\begin{equation}\label{eq2.6}
\begin{split}
\|I_2(\phi_N,\phi_N,t)\|_{H^s(\R)} &\gtrsim |t|\Big(\int_{|\xi|\sim {\text{o}}(1)}(1+|\xi|^2)^s |\varphi(\xi)|^2d\xi\Big)^{\frac12}\\
&\gtrsim |t|\Big(\int_{|\xi|\sim {\text{o}}(1)}|\xi|^2d\xi\Big)^{\frac12}\\
&\geq C_0.
\end{split}
\end{equation}

By construction, $\|\phi_N\|_{H^s}\to 0$ for any $s<0$, therefore \eqref{eq2.6} ensures that for any fixed $t>0$, the application $u_0\mapsto I_2(u_0, u_0, t)$ is not continuous at the origin from $H^s(\R)$ to even $\mathcal{D}'(\R)$.

Now, our idea is to prove that the discontinuity of $u_0\mapsto I_2(u_0, u_0, t)$ at the origin implies the discontinuity of the flow-map $u_0\mapsto u(t)$.

From Theorem \ref{well-posed}, there exist $T>0$ and $\epsilon_0 >0$ such that for any $|\epsilon|\leq \epsilon_0$, any $\|h\|_{L^2(\R)}\leq 1$ and $0\leq t\leq T$, one has
\begin{equation}\label{eq2.7}
u(\epsilon h, t) = \epsilon S(t)h + \sum_{k=2}^{+\infty} \epsilon^k I_k(h^k, t),
\end{equation}
where $h^k:= (h, h, \cdots, h)$, $h^k\mapsto I_k(h^k, t)$ is a $k$-linear continuous map from $L^2(\R)^k$ into $C([0, T]; L^2(\R))$ and the series converges absolutely in $C([0, T]; L^2(\R))$.

From \eqref{eq2.7}, we have that
\begin{equation}\label{eq2.8}
u(\epsilon \phi_N, t) -\epsilon^2I_2(\phi_N, \phi_N, t) = \epsilon S(t)\phi_N + \sum_{k=3}^{+\infty} \epsilon^k I_k(\phi_N^k, t).
\end{equation}

Also, we have that
\begin{equation}\label{eq2.9}
\|S(t)\phi_N\|_{H^s(\R)}\leq \|\phi_N\|_{H^s(\R)} \sim N^s
\end{equation}
and
\begin{equation}\label{eq2.10}
\Big\|\sum_{k=3}^{+\infty} \epsilon^k I_k(\phi_N^k, t)\Big\|_{L^2(\R)}\leq \Big(\frac{\epsilon}{\epsilon_0}\Big)^3 \sum_{k=3}^{+\infty} \epsilon_0^k \|I_k(\phi_N^k, t\|_{L^2(\R)}
\leq C\epsilon^3.
\end{equation}

Therefore, from \eqref{eq2.8} in view of \eqref{eq2.9} and \eqref{eq2.10} we get, for any $s<0$,
\begin{equation}\label{eq2.11}
\sup_{t\in [0, T]}\|u(\epsilon\phi_N, t) -\epsilon^2 I_2(\phi_N, \phi_N, t)\|_{H^s(\R)}\leq O(N^s) +C\epsilon^3.
\end{equation}

Now, if we fix $0<t<1$, take $\epsilon$ small enough and then $N$ large enough, and take an account of  \eqref{eq2.6}; the estimate \eqref{eq2.11} yields that $\epsilon^2 I_2(\phi_N, \phi_N, t)$ is a good approximation of $u(\epsilon\phi_N, t)$ in $H^s(\R)$ for any $s<0$.

If we choose $\epsilon \ll 1$, from \eqref{eq2.8}, \eqref{eq2.9} and \eqref{eq2.10}, we get 
\begin{equation}\label{eq2.12}
\begin{split}
\|u(\epsilon\phi_N, t)\|_{H^s(\R)}&\geq \epsilon^2\|I_2(\phi_N, \phi_N, t\|_{H^s(\R)} -\epsilon\|S(t)\phi_N\|_{H^s(\R)}-\sum_{k=3}^{+\infty} \epsilon^k \|I_k(\phi_N^k, t\|_{H^s(\R)}\\
&\geq C_0\epsilon^2 -C_1\epsilon^3 - C\epsilon N^s\\
&\geq \frac{C_0}{2}\epsilon^2 -C\epsilon N^s.
\end{split}
\end{equation}

If we fix the $\epsilon \ll 1$ chosen earlier and choose $N$ large enough, then for any $s<0$, the estimate \eqref{eq2.12} yields,
\begin{equation}\label{eq2.13}
\|u(\epsilon\phi_N, t)\|_{H^s(\R)}\geq \frac{C_0}{4} \epsilon^2.
\end{equation}

Note that, $u(0, t) \equiv 0$ and $\|\phi_N\|_{H^s(\R)}\to 0$ for any $s<0$. Therefore, taking $N\to \infty$ we conclude that the flow-map $u_0\mapsto u(t)$ is discontinuous at the origin from $H^s(\R)$ to $C([0, 1]; H^s(\R)$, for $s<0$. Moreover, as $\phi_N \rightharpoonup 0$ in $L^2(\R)$, we also have that the flow-map is discontinuous from $L^2(\R)$ equipped with its weak topology inducted by  $H^s(\R)$ with values even in $\mathcal{D}'(\R)$.
\end{proof}

\begin{remark} In the periodic case, i.e., for $x \in \T$, there is analytical well-posedness result for given data in Sobolev spaces without zero Fourier mode ( i.e., with zero $x$-mean) $H^s(\T)$, $s\geq 0$, see \cite{DR}. Now, for $N\gg 1$, if we define a sequence of functions $a_n$, by
\begin{equation}\label{ex-5}
a_n =\begin{cases} 1, \quad |n|\sim N\\
	                 0, \quad {\mbox{otherwise}},
	   \end{cases}
\end{equation}
	                 and $\phi_N$ by  $\widehat{\phi_N}(n) = a_n$, then clearly
 $\|\phi_N\|_{L^2(\T)} \sim 1$ and  $\|\phi_N\|_{H^s(\T)} \to 0$, for any $s <0$.
If we proceed with the calculations exactly as above considering the Sobolev spaces $H^s(\T)$ without zero Fourier mode, we can obtain a similar ill-posedness result for $s<0$, in the periodic case too.
\end{remark}

\noindent{\bf Acknowledgment:} The author would like to thank N. Tzvetkov for discussions on this problem.


\end{document}